\documentclass{article}
\usepackage{graphicx} % Required for inserting images
\usepackage[utf8]{inputenc}
\usepackage{amssymb,amsthm,amsmath,amstext,amscd}
\usepackage{xcolor}
\usepackage[normalem]{ulem}
\usepackage{tikz-cd}
\usepackage{indentfirst}
\usepackage{hyperref}
\usepackage{verbatim}

\allowdisplaybreaks
\numberwithin{equation}{section}\theoremstyle{plain}
\newtheorem{theorem}{Theorem}

\newtheorem{corollary}[theorem]{Corollary}
\newtheorem{proposition}[theorem]{Proposition}
\newtheorem{lemma}[theorem]{Lemma}
\newtheorem{example}[theorem]{Example}
\newtheorem{remark}[theorem]{Remark}

\newcommand{\der}{{\rm Der}}
\newcommand\I{\mathbb I}

\DeclareMathOperator{\Aut}{Aut}
\DeclareMathOperator{\LND}{LND}
\DeclareMathOperator{\Inn}{Inn}

\title{On the isotropy of differential Ore extensions}
\author{Rene Baltazar, Leonardo Duarte Silva, Grasiela Martini}
\date{}

\begin{document}

\maketitle

\begin{abstract}
%Let $\Bbbk$ be an algebraically closed field of characteristic zero and 
Let $A_h=\Bbbk[x][t;d]$ be the differential Ore extension. We study the action of the automorphism group $\Aut(A_h)$ on the derivations $\der(A_h)$ and explicitly describe, using Nowicki’s decomposition of the derivations of $A_h$, the isotropy groups of this action.
More precisely, we first obtain an explicit description of $\Aut(A_h)$ for $\deg(h)\geq 1$;
then, we determine the isotropy groups of the derivations of the form $D=ad_w+\Delta_{s(x)}$, which are all the derivations in the square-free case, that is, when $\gcd(h,h')=1$.
After, in the singular case $\gcd(h,h')\neq 1$, where special derivations of type $E_H$ appear, we show that the isotropy problem is ruled by a suitable localization and by the element $w^\ast=w+\psi^{-1}H$ with $\psi=\gcd(h,h')$.
This allows us to obtain a general criterion for the isotropy of a derivation $D=ad_w+E_H+\Delta_{s(x)}$.
Finally, explicit results are provided to illustrate the new phenomena that arise in this setting.

\end{abstract}

\maketitle

\section{Introduction}

Let $\Bbbk$ be an algebraically closed field of characteristic zero.
For a $\Bbbk$-algebra $R$, we denote by $\Aut(R)$ and $\der(R)$ the groups of $\Bbbk$-automorphisms and $\Bbbk$-derivations of $R$, respectively.
Note that $\Aut(R)$ acts by conjugation over $\der(R)$: given $\delta \in \der(R)$ and $\rho \in \Aut(R)$, then $\rho \delta \rho^{-1} \in \der(R)$.
The \emph{isotropy subgroup of $\delta$}, with respect to this group action, is defined
as
$$\Aut_\delta(R):=\{ \rho \in \Aut(R) \ \vert \ \rho \delta = \delta\rho \}. $$

The structure and behavior of an isotropy subgroup remain largely unexplored.
However, aspects of symmetry and connections with several areas are being investigated: isotropy groups play a fundamental role across various areas of mathematics, appearing naturally in the study of symmetries of foliations, the structure of Hochschild cohomology, and the invariants under group actions in Algebra and Geometry \cite{PanBalta, PanMendes}. 

The differential Ore extensions arise from an interesting classical result: an Ore extension over the polynomial algebra $\Bbbk[x]$ is either a quantum plane, a quantum Weyl algebra, or an infinite-dimensional associative algebra $A_h$ which is generated by the elements $x$ and $t$ satisfying $tx-xt = h$, where $h \in \Bbbk[x]$ (for more details, see \cite{SBVA2025, SBVA2024}).
Since the isotropy groups of a quantum plane and a quantum Weyl algebra have already been investigated in \cite{SBVA2025} and \cite{SBVA2024}, the present work focuses on the isotropy groups of $A_h$.

In the Section \ref{section2}, we explore some basic concepts and important preliminary results needed for this work. This section gives an overview of the algebraic structures and main definitions used in our study, as well as key lemmas and propositions that will be important in the later sections.

In Section \ref{section3} we provide an explicit description of the automorphism group
$\Aut(A_h)$ for differential Ore extensions $A_h$ with $\deg(h)\ge1$.
Using a convenient change of variable, we reduce the study to the case where the
defining polynomial $h \in \Bbbk [x]$ is monic and has a suitable coefficient zero (the one that accompanies the term of degree
$\deg(h)-1$), which allows us to apply and extend well-known classification results.

Section \ref{section4} is devoted to the study of the action of $\Aut(A_h)$ on the space of
derivations and, in particular, to the explicit description of isotropy
subgroups in the square-free case, that is, under the assumption $\gcd(h,h')=1$.
Using the canonical decomposition of derivations due to Nowicki \cite[Thm. 10.1]{Nowicki}, we analyze how
each component behaves under conjugation and show that the spaces of the inner
derivations and the differential derivations of type $\Delta_{s(x)}$ are two independent
$\Aut(A_h)$-submodules of $\der(A_h)$.
It allows us to compute the isotropy group of a sum of two derivations as the intersection of the isotropy groups of each term in the sum, providing a structural framework that will be used throughout the paper.

Section \ref{section5} is dedicated to the singular case, that is, when $\gcd(h,h')\neq 1$.
In this setting, Nowicki’s decomposition \cite[Thm. 11.2]{Nowicki} shows that a derivation of $A_h$ may contain, in addition to the inner and the differential parts, an additional special component $E_H$, which substantially changes the isotropy problem.
In contrast to the case treated previously, the derivations of the type $E_H$ are not stable under conjugation: indeed, it may produce simultaneously inner and differential terms.
To handle this phenomenon, we pass to a suitable Ore localization and combine the inner and special parts into the localized element $w^\ast=w+\psi^{-1}H$. This leads to a general criterion for the isotropy of derivations of the form $D=ad_w+E_H+\Delta_{s(x)}$.
In the end, we provide explicit results that illustrate the new phenomena that arise in the singular case.

\section{Preliminaries}\label{section2}

For a nonzero polynomial $h(x) = \sum_{i=0}^N a_i x^i \in \Bbbk[x]$, the \emph{degree of $h(x)$} will be denoted as $\deg(h(x))=N \geq 0$.
It is also convenient to define the degree of the zero polynomial as $-\infty$.

Let $d=h(x) \, \partial_x$ be a derivation in $\Bbbk[x]$ with $\deg(h(x)) \geq 0$.
We denote by $A_h$ the classical differential Ore extension $\Bbbk[x] [t;d]$: the $\Bbbk$-algebra generated by $x$ and $t$ subject to the relation $tx = xt+h(x)$.
In particular, if $\deg(h(x))=0$, then $A_h \simeq A_1$, where $A_1$ denotes the (first) Weyl algebra. For each integer $n\geq 1$, we denote by $\mathbb{G}_n=\{a\in\Bbbk \mid a^n=1\}$ the group of all $n$-th roots of unity in $\Bbbk$. For convenience, we also set $\mathbb{G}_0=\Bbbk^\ast.$

\subsection{The differential algebra $A_h$}
The study of automorphisms of the algebra $A_h$ followed an approach analogous to that adopted in some classical algebras: $\Bbbk[x,y]$ (see \cite{Re1968}), $A_1$ (see \cite{Dixmier}) and $\Bbbk\langle x, y \rangle$ (see \cite{SC20}); in which the classification of locally nilpotent derivations plays a central role. Indeed, as in the famous Rentschler's Theorem \cite{Re1968}, there is an explicitation of locally nilpotent derivations in $A_h$ with $h\in \Bbbk[x]\setminus\Bbbk$.
For example, in both Kaygorodov, Lopes and Mashurov's work \cite{ISF2021}, as well as Benkart, Lopes and Ondrus' paper \cite{BLO2015}, they exhibit this family explicitly and so, as a consequence, a characterization of the automorphisms of this algebra is obtained.

Note that given $h \in \Bbbk[x]\setminus\Bbbk$, there are derivations in $A_h$ that are analogous to the locally nilpotent derivations $v(x)\partial_y$ of $k[x,y]$ and $ad_{v(x)}$ of the (first) Weyl algebra $A_1$.
Indeed, let $g(x) \in \Bbbk[x]$ and consider the derivation $D_{g(x)}$ of $A_h$ given by $D_{g(x)}(x)=0$ and $D_{g(x)}(t)=g(x)$.
We denote by $\LND(A_h)$ the set of all locally nilpotent derivations in $A_h$. 
It is easy to see that $D_{g(x)} \in \LND(A_h)$. 
Furthermore, locally nilpotent derivations play an important role to obtain the automorphisms of an algebra.
Indeed, for a $\partial \in \LND(A_h)$, there exists a well-defined map exp$(\partial): A_h \to A_h$, with exp$(\partial)(a)=\sum_{k\geq 0} \frac{\partial^k(a)}{k!}$, which is an automorphism of $A_h$.
Moreover, the set of these automorphisms generates a normal subgroup of all automorphisms of the algebra. 
For the algebra $A_h$, both $\LND(A_h)$ and $\Aut(A_h)$ are determined in the following results:

\begin{proposition}\cite[Prop. 2]{ISF2021}\label{prop_localnilp}
Let $h \in \Bbbk[x] \setminus \Bbbk$.
Then  $LND(A_h)=\{D_{g(x)} \mid g(x) \in \Bbbk[x]\}$.
\end{proposition}

\begin{theorem} \label{cod_aut} \cite[Thm. 8.3]{BLO2015} Consider $A_h$, where $\deg(h(x))=N \geq 1$, and $\mathbb{P}=\{ (a,b)\in \Bbbk^*\times \Bbbk \ | \ h(ax+b)=a^Nh(x)\}$.
Then, $\Aut (A_h)=\Bbbk[x]\rtimes \tau_{\mathbb{P}},$ where $ \tau_{\mathbb{P}}:=\{\tau_{a,b} \ | \ (a,b)\in\mathbb{P}\}$.
If $\rho\in \Aut(A_h)$, then there exist $(a,b)\in \mathbb{P}$ and $r(x)\in \Bbbk[x]$ such that $\rho=\sigma_{r(x)}\circ\tau_{a,b}$, where
    $\sigma_{r(x)}(x)=x$, $\sigma_{r(x)}(t)=t+r(x)$, $\tau_{a,b}(x) = ax+b$ and $\tau_{a,b}(t) = a^{N-1}t$.
\end{theorem}

\medbreak

The following result establishes necessary and sufficient conditions for an isomorphism between algebras of this work:

\begin{theorem}\label{teo_isoA_h}\cite[Thm. 8.2]{BLO2015}
Let $g,h \in \Bbbk[x]$.
$A_h$ is isomorphic to $A_g$ if and only if there exist $\alpha, \beta, \gamma \in \Bbbk$, with $\alpha \gamma \neq 0$ such that $\gamma g(x) = h(\alpha x + \beta)$.
    In particular, if $A_h$ is isomorphic to $A_g$, then $g$ and $h$ have the same degree.
 \end{theorem}

The following result, used in subsequent chapters, explains the group of automorphisms according to the characteristics of $h(x)$, and so the algebra $A_h$.

\begin{theorem}\cite[Thm. 8.6]{BLO2015}\label{8_6_Samuel}
Assume that $h$ has $k$ distinct roots in $\Bbbk$.
\begin{itemize}
    \item[i)] If $k=1$ (and so $\lambda \in \Bbbk$ is the unique root of $h$), then $\mathbb{P} = \{(\alpha, (1-\alpha) \lambda) \ | \ \alpha \in \Bbbk \}, \tau_{\mathbb{P}} \simeq \Bbbk^\ast$, and $\Aut(A_h) = \Bbbk[x] \rtimes \Bbbk^\ast$
    \item[ii)] If $k > 1$, then $\tau_{\mathbb{P}} = \langle \tau_{\alpha, \beta} \rangle$, where $\alpha^k=1$ or $\alpha^{k-1}=1$, and $\Aut(A_h) = \Bbbk[x] \rtimes \langle \tau_{\alpha, \beta} \rangle \simeq \Bbbk[x] \rtimes  \mathbb{G}_\ell$, for some $\ell$ such that $\ell$ divides $k$ or $k-1$.
\end{itemize}
\end{theorem}

We also emphasize the following result by Nowicki, which characterizes the derivations in $A_h$ when $h(x)$ is square-free, that is, $gcd(h,h')=1$.

\begin{theorem}\label{nowicki_mdc1}\cite[Thm. 10.1]{Nowicki}
Consider the algebra $A_h$ with $\deg(h)\geq 1$. Assume that $gcd(h,h')=1$. Then every derivation $D$ of $A_h$ has a unique decomposition
\begin{equation}\label{eq_Dmdc1}
    D=ad_w+\Delta_{s(x)},
\end{equation}
where $ad_w$ is an inner derivation of $A_h$, and  $\Delta_{s(x)}$ is the derivation of $A_h$ given by $\Delta_{s(x)} (x)=0$ and $\Delta_{s(x)} (t)=s(x)$, for $s(x)\in \Bbbk[x]$ and $\deg(s)<\deg(h)$.
Moreover, $D$ is inner if and only if $s(x)=0$.
\end{theorem}

\subsection{Properties of $A_h$}

We consider some formulas for the product of elements in $A_h$.
The preliminary results developed here also hold in the case $\deg(h(x))=0$, although we will not use this case in the following sections.
For $f(x) \in \Bbbk [x]$, we denote by $f^\prime(x)=\left(f(x)\right)^\prime$ the usual derivative, that is, if $f(x)=\sum_{i=0}^n a_i x^i$, then $f^\prime(x) = \sum_{i=1}^n i a_i x^{i-1}$.

\begin{lemma}\label{lematx^n} In $A_h$, for any $n \geq 0$, it holds
\begin{equation}\label{lematx^n_eq}
tx^n=x^nt+(x^n)'h(x)=x^nt+nx^{n-1} h(x).
\end{equation}

\end{lemma}
\begin{proof}
    We prove this by induction on $n$.
    For $n=0$ and $n=1$ is clear.
    Assume valid for $n\geq 0$, then
    \begin{eqnarray*}
    tx^{n+1}&=& tx^nx\\
    &=& (x^nt+nx^{n-1} h(x))x\\
    &=& x^ntx+nx^{n-1} h(x)x\\
    &=& x^n(xt+h(x)) + nx^nh(x)\\
    &=& x^{n+1}t + x^nh(x) + nx^nh(x)\\
    &=& x^{n+1}t + (n+1)x^{n}h(x).
       \end{eqnarray*}
       Therefore,  $tx^n=x^nt+nx^{n-1} h(x)$, for all $n\geq 0$.
\end{proof}

Thus, it is possible to compute the product of $t$ with any polynomial of $\Bbbk[x]$: 

\begin{equation}\label{comm_t_fx_eq}
t f(x) = f(x) t + f'(x)h(x).
\end{equation}

In general, the product of $t$ with any element of $A_h$ is more complicated.
But sometimes it is sufficient to know the behavior of this product with respect to degree in $t$.
Recall that for $g \in A_h, g \neq 0$, there is a unique decomposition $g = \sum_{i=0}^n g_i(x)t^i$, where $g_i \in \Bbbk[x]$ and $g_n\neq 0$.
In that case, we set $\deg_t(g):=n$.
Note that $\deg_t(g)=0$ if and only if $g \in \Bbbk[x]$.
It is also convenient to define $\deg_t(0)=-\infty$.

\begin{corollary}\label{comm_t_g}
    Let $g \in A_h$. Then $t g = g t + h(x)\tilde{g}$, where $\deg_t(\tilde{g}) \leq \deg_t(g)$.
    In particular, $\deg_t( t g ) = \deg_t (g) + 1$, and $\deg_t(\tilde{g})= \deg_t(g)$ if and only if $g_n(x) \in \Bbbk[x] \backslash \Bbbk$, where $g = \sum_{i=0}^n g_i(x) t^i$.
\end{corollary}

\begin{proof}
    Write $g = \sum_{i=0}^n g_i(x)t^i$.
    Then,
    \begin{eqnarray*}
        t g & = & t \left( \sum_{i=0}^n g_i(x)t^i \right) \\
        & = & \sum_{i=0}^n t g_i(x) t^i \\
        & \stackrel{\eqref{comm_t_fx_eq}}{=} & \sum_{i=0}^n (g_i(x) t + g_i'(x) h(x)) t^i \\
        & = & \sum_{i=0}^n g_i(x)t^{i+1} + \sum_{i=0}^n h(x)g_i'(x)t^i \\
        & = & \left( \sum_{i=0}^n g_i(x)t^i \right) t + h(x) \left( \sum_{i=0}^n g_i'(x)t^i \right) \\
        & = &  gt + h(x) \tilde{g},
    \end{eqnarray*}
    where $\tilde{g}= \sum_{i=0}^n g_i'(x)t^i$.
\end{proof}

\medbreak

Next, we also use the standard bracket notation:
let $R$ be an algebra; for $a,b \in R$, then $[a,b]:=ab-ba$.
In particular, the inner derivation $ad_a : R \to R$ is is defined as $ad_a(b)=ab-ba = [a,b]$.
Thus, sometimes we write $ad_a = [a, -] \in \der(R).$

\begin{lemma}\label{lematix}
For every $i\ge 1$, there exists $u_i\in A_h$ such that
\[
[t^i,x]=i\,h(x)\,t^{i-1}+h(x)u_i
\quad\text{and}\quad
\deg_t(u_i)<i-1.
\]
\end{lemma}

\begin{proof}
For $i=1$ is immediate, since $[t,x]=h(x)$.
Assume that $[t^i,x]=i\,h(x)\,t^{i-1}+h(x)u_i$ for some $u_i\in A_h$ with $\deg_t(u_i)<i-1$. Using
\[
[t^{i+1},x]=t[t^i,x]+[t,x]t^i
\]
and $th(x)=h(x)(t+h'(x))$, we obtain
\[
[t^{i+1},x]
=t\bigl(i\,h(x)\,t^{i-1}+h(x)u_i\bigr)+h(x)t^i
=(i+1)h(x)t^i+h(x)u_{i+1},
\]
for some $u_{i+1}\in A_h$ with $\deg_t(u_{i+1})<i$. This concludes the induction.
\end{proof}

\medbreak

\begin{lemma}\label{lema_rhotnai}
Let $r(x)\in\Bbbk[x]$, then 
    $(t+r(x))^i=t^i+ir(x)t^{i-1}+g_i$, for all $i \geq 1$, where $g_i\in A_h$ and $\deg_{t}(g_i)\leq i-2$.
\end{lemma}
\begin{proof}
    We proof by induction on $i$.
    For $i=1$, it is clear.
    For $i=2$, $(t+r(x))^2=t^2+2r(x)t+r'(x)h(x)+(r(x))^2$, with $g_2=r'(x)h(x)+(r(x))^2\in\Bbbk[x]$.
    
    Recall from Corollary \ref{comm_t_g} that $tg=gt+h(x)\tilde{g}$, where $\deg_t(\tilde{g}) \leq \deg_t(g)$.
    Thus, assuming that the result holds for $i \geq 1$, 
    \begin{align*}
    (t+r(x))^{i+1} = & (t+r(x)) (t+r(x))^i \\
    = &  (t+r(x))(t^i + i r(x) t^{i-1} + g_i) \\
    = & t^{i+1} + i t r(x) t^{i-1} + t g_i+ r(x) t^i + i (r(x))^2 t^{i-1} +r(x) g_i \\
\stackrel{\eqref{comm_t_fx_eq}}{=} & t^{i+1} + i( r(x) t + r'(x)h(x)) t^{i-1} + g_i t + h(x) \tilde{g_i} + r(x)t^i \\
    & + i (r(x))^2 t^{i-1} + r(x) g_i \\
    = & t^{i+1} + (i+1)r(x) t^i + g_{i+1},
    \end{align*}
    where $g_{i+1} = i (r'(x)h(x) + (r(x))^2) t^{i-1} + g_i t + h(x) \tilde{g_i} +  r(x) g_i$, and, therefore, $\deg_t(g_{i+1}) \leq i-1$.
\end{proof}

\section{The group $\Aut(A_h)$, where $\deg(h)  \geq 1$}\label{section3}

In this section, we use a convenient change of variable to obtain an explicit characterization of the group of automorphisms of $A_h$, described in \cite{BLO2015}.
To achieve this goal, the following lemma plays a fundamental role.

\begin{lemma}\label{lemma_caso_bzeroVeloso}\cite[Lem. 18]{BiaVeldan} Let $h(x)$ be a monic polynomial  of degree $N \geq 2$ in $\Bbbk[x]$  with at least one nonzero root and such that the coefficient of the term of degree $N-1$ is zero. If $a,b\in \Bbbk$, $a\neq 0$,  then $h(ax+b)=a^N h(x)$ if, and only if, $b=0$, $a$ is a $\ell$-th root of unity and $h(x)=x^i g(x^\ell)$, $\ell \leq N$, with $g(x)\in \Bbbk[x]$ monic.
\end{lemma}

Since $\Bbbk$ is an algebraically closed field of characteristic zero, in the following we consider the polynomial $h(x) = x^N$ with $N\geq 1$, and it is immediate to conclude that part of this result also holds.

\begin{lemma}\label{lema_xN_bzero}
   Let $h(x)=x^N$, $N \geq 1$, and $a, b \in \Bbbk, a \neq 0$.
    Then, $h(ax+b) = a^N h(x)$ if and only if $b=0.$
\end{lemma}

\medbreak

\begin{example}
    Let $h(x)=x^N$, with $N \geq 1$.
    By Lemma \ref{lema_xN_bzero} and Theorem \ref{cod_aut}, we get $\Aut(A_{x^N}) \cong \Bbbk[x]\rtimes \tau_{\mathbb{P}} \cong \Bbbk[x]\rtimes \Bbbk^*$.
\end{example}

\begin{proposition}
    Let $h(x)\in \Bbbk[x]$, $\deg(h)\geq 1$.  Then, there exists a monic polynomial $h^*\in \Bbbk[x]$ such that $A_h\simeq A_{h^*}$ and 
    $$\Aut(A_{h^*})=\{\sigma_{r(x)}\circ\tau_{a} \ | \ h^*(ax)=a^Nh^*(x)\},$$
where $\tau_a := \tau_{a,0}$.
\end{proposition}
\begin{proof} Let $h(x)=\sum_{i=0}^{N}c_ix^i$, $\deg(h)=N\geq 1$, and consider
     $h^*(x)=\frac{1}{c_N}h(x-\frac{c_{N-1}}{N c_N})$. By Theorem  \ref{teo_isoA_h} $A_h\simeq A_{h^*}$ and, in this case, $h^*$ is a monic polynomial and has the coefficient of the term of degree $N-1$ equal to zero. Therefore, we have the following cases, for $a,b\in\Bbbk$, $a\neq 0$:
\begin{itemize}
    \item[i)] If $\deg(h)=1$, then $h^*(x)=x$. Thus, $h^*(ax+b)=ah^*(x)$ if and only if $b=0$.

    \item[ii)] If $\deg(h)\geq 2$ with at least one nonzero root, then $h^*(ax+b)=ah^*(x)$ if and only if $b=0$, by Lemma  \ref{lemma_caso_bzeroVeloso}.

     \item[iii)] If $\deg(h)=N\geq 2$ and has all roots equal to zero, then $h^*(x)=x^N$, and so $h^*(ax+b)=ah^*(x)$ if and only if $b=0$, by Lemma \ref{lema_xN_bzero}.
\end{itemize}

     Therefore, by Theorem \ref{8_6_Samuel}, 
     $\Aut(A_{h^*})=\{\sigma_{r(x)}\circ\tau_{a}| h^*(ax)=a^Nh^*(x)\}.$
\end{proof}

\medbreak 

In light of the above proposition, from now on we always consider a polynomial $h(x)=\sum_{i=0}^N c_i x^i$, where $N \geq 1$, $c_N=1$ and $c_{N-1}=0$, and so 
\begin{equation}\label{aut_Ah}
    \Aut(A_{h})=\{\sigma_{r(x)}\circ\tau_{a} \ | \ h(ax)=a^Nh(x)\}.
\end{equation}

\medbreak

In what follows, we will use the following notations:
\begin{itemize}
    \item for $k, s \in \mathbb{Z},$ where $k \leq s$, let $\I_{k,s} = \{ x \in \mathbb{Z} \ | \ k \leq x \leq s\}$;
    \item for $f=\sum_{i=0}^{N}c_ix^i\in \Bbbk[x]$, let $\mathcal{S}_f = \{ i \ | \ c_i \neq 0\}$.
\end{itemize}

\begin{lemma}\label{lemma_raizdaunidade}
Let $f=\sum_{i=0}^{N}c_ix^i\in \Bbbk[x]$, $N\geq 0$, $a\in \Bbbk^\ast$ and $\ell\in\mathbb{Z}$.
Then, $a^\ell f(x) = f(ax)$  if and only if $a\in \mathbb{G}_{|i-\ell|}$, for all $i \in \mathcal{S}_f.$
In particular, $a \in \mathbb{G}_{|\deg(f)-\ell|}$.
\end{lemma}
\begin{proof}
    Let $f=\sum_{i=0}^{N}c_ix^i \in \Bbbk[x]$.
    Then, \begin{eqnarray*}
        a^\ell f(x) = f(ax) & \Longleftrightarrow & \sum_{i=0}^{N}a^{\ell}c_ix^i= \sum_{i=0}^{N}a^ic_ix^i\\
        & \Longleftrightarrow & a^{\ell}c_i= a^ic_i, \ \forall i \in \I_{0,N}\\
        & \Longleftrightarrow & a^{\ell}=a^i, \ \forall i \in \mathcal{S}_f \\
        & \Longleftrightarrow & a^{|i-\ell|}=1, \ \forall i \in \mathcal{S}_f \\
        & \Longleftrightarrow & a\in \mathbb{G}_{|i-\ell|}, \ \forall i \in \mathcal{S}_f.
    \end{eqnarray*}
    In particular, $N=\deg(f) \in \mathcal{S}_h$ and, therefore, $a \in \mathbb{G}_{|\deg(f)-\ell|}$.
\end{proof}

We will use the above lemma for $\deg(h)=N \geq \ell \geq 0$.
In that case, for $h(x)=\sum\limits_{i=0}^{N}c_ix^i $, $N\geq 1$, we have $\Aut(A_h)=\Bbbk[x]\rtimes \tau_{\mathbb{P}},$ where 
\begin{equation}\label{tau_P}
\tau_{\mathbb{P}}:=\ \left\{\tau_{a} \ | \ a\in \cap_{i \in \mathcal{S}_h}\mathbb{G}_{N-i}\right\} \subseteq \Bbbk^{\ast}.
\end{equation}

Thus, we can precisely determine $\tau_{\mathbb{P}}$: it is finite if and only if $\cap_{i \in \mathcal{S}_h}\mathbb{G}_{N-i}$ is finite if and only if there is $c_i\neq 0$, $i\in \I_{0,N-1}$, \emph{i.e.}, there exists $i \in \mathcal{S}_h \backslash \{N\}$.
In that case, $\cap_{i \in \mathcal{S}_h}\mathbb{G}_{N-i}=\mathbb{G}_{gcd(N-i_1, N-i_2, ..., N-i_k)},$ where $\mathcal{S}_h = \{ i_1, i_2, ..., i_k, N \}.$
Otherwise, $\tau_{\mathbb{P}} = \Bbbk^\ast$, and so it is infinity. 

\medbreak

Therefore, we obtain a more explicit characterization for $\Aut(A_h)$ than Theorem \ref{8_6_Samuel}, and conclude this section by summarizing the above discussion as follows.

\begin{corollary}\label{carac_Aut}
Let $h = \sum_{i=0}^{N}c_ix^i \in  \Bbbk[x]$, $\deg(h) = N \geq 1$.
Then,
\begin{itemize}
    \item[i)] $\Aut(A_h) = \Bbbk[x] \rtimes \Bbbk^\ast$, if $h(x)=x^N$; 
    \item[ii)] $\Aut(A_h) = \Bbbk[x] \rtimes \left(\cap_{i \in \mathcal{S}_h}\mathbb{G}_{N-i}\right) = \Bbbk[x] \rtimes \mathbb{G}_{gcd(N-i_1, N-i_2, ..., N-i_k)},$ where $\mathcal{S}_h = \{i_1, i_2, ..., i_k, N \}$, otherwise.
\end{itemize}
\end{corollary}

\medbreak 

By \eqref{aut_Ah}, we know that if $\rho \in \Aut(A_h)$, $\deg(h)\geq 1$, then $\rho = \sigma_{r(x)} \circ \tau_a$, for some $r(x) \in \Bbbk[x]$ and $a \in \Bbbk^\ast$.
However, in what follows, sometimes we need to consider the composition in the inverse order $\tau_a \circ \sigma_{r(x)}$, also $\rho^{-1}$ and the powers of $\rho$.
Note that $\tau_a \circ \sigma_{r(x)} = \sigma_{a^{1-N} r (ax)} \circ \tau_a$ and so $\rho^{-1} = \tau_{a^{-1}} \circ \sigma_{-r(x)} = \sigma_{-a^{N-1}r(a^{-1}x)} \circ \tau_{a^{-1}}$.
For the powers of $\rho$, we have the following lemma.

\begin{lemma}\label{rho_n}
 Let $\rho= \sigma_{r(x)}\circ \tau_a\in \Aut(A_h)$, $\deg(h)\geq 1$. 
Then 
$\rho^n = \sigma_{R_n(x)} \circ \tau_{a^n}$, for all $n\geq 1$, where $R_n(x) = \sum_{i=0}^{n-1} a^{i(1-N)} r(a^i x)$. 
    \end{lemma}
    \begin{proof}
 We prove it by induction on $n$.
  For $n=1$ is clear.    
Assume that it holds for $n\geq 1$.
Since $\tau_a \circ \sigma_{r(x)} = \sigma_{a^{1-N} r (ax)} \circ \tau_a$, then
    \begin{eqnarray*}
        \rho^{n+1}&=& \rho\circ\rho^n\\
        &=& \sigma_{r(x)}\circ\tau_a\circ\sigma_{R_n(x)}\circ\tau_{a^n}\\
        &=& \sigma_{r(x)}\circ\sigma_{a^{1-N}R_n(ax)}\circ\tau_a\circ\tau_{a^n}\\
        &=& \sigma_{r(x)+ \sum_{i=0}^{n-1} a^{(i+1)(1-N)} r(a^{i+1} x)}\circ\tau_{a^{n+1}}\\
        &=& \sigma_{ \sum_{i=0}^{n} a^{i(1-N)} r(a^i x)}\circ\tau_{a^{n+1}}\\
        &=& \sigma_{R_{n+1}(x)} \circ \tau_{a^{n+1}},
    \end{eqnarray*}
    and, therefore, the result follows.
     \end{proof}

\section{The square-free case}\label{section4}

Consider a  square-free polynomial $h\in \Bbbk[x]$, where $\deg(h)\geq 1$.
In this case, it is known by \cite{Nowicki} (see Thm. \ref{nowicki_mdc1}) that every derivation $D$ in $A_h$ has a unique decomposition $D=ad_w + \Delta_{s(x)}$, where $ad_w$ is the inner derivation of $A_h$ determined by $w \in A_h$, and $\Delta_{s(x)} \in \der(A_h)$ is given by  $\Delta_{s(x)} (x)=0$ and $\Delta_{s(x)} (t)=s(x)$, where $s(x) \in \Bbbk[x]$ with $\deg(s) < \deg(h)$. In this section, we characterize the isotropy group $\Aut_D(A_h)$, and determine it explicitly in some interesting cases. 

\medbreak 

In order to understand $\Aut_{ad_w + \Delta_{s(x)}}(A_h)$, we start with more general considerations.

\begin{lemma}\label{lema_interna_geral}
    
Let $ R $ be a $ \Bbbk$-algebra, $ ad_w $ an inner derivation of $R$, and $ \varphi \in \Aut(R) $. Then $ \varphi \in \Aut_{ad_w}(R) $ if and only if $ \varphi(w) - w \in Z(R)$ .
\end{lemma}
\begin{proof} 
Since $ \varphi $ is an automorphism, it holds $ \varphi([w,u]) = [\varphi(w), \varphi(u)]$, for all $ u \in R $.
Moreover, $ \varphi \in \Aut_{ad_w}(R) $ if and only if $ \varphi([w,u]) = [w, \varphi(u)]$.
At last, $[\varphi(w), \varphi(u)] = [w, \varphi(u)]$ if and only if $\varphi(w) -w \in Z(R)$.
\end{proof}

However, it is already known that the center of the algebra $A_h$ is the field $\Bbbk$ (see \cite[Sec. 5]{BLO2015}), thus we directly have the following result:

\begin{corollary} \label{eq_groupisot_interna} Let $h \in \Bbbk[x]$ with $\deg(h)\geq 0$. Consider $w \in A_h$ and the inner derivation $ad_w$ of $A_h$.
    Then, %the isotropy group of $D$ is given by
    \begin{equation*}
    \Aut_{ad_w}(A_h)=\{\rho \in \Aut(A_h) \ | \ \rho(w)-w\in \Bbbk \}.
    \end{equation*}
    \end{corollary}

\medbreak 

If $\deg(h)=0$, all derivations are inner since $A_h\simeq A_1$ is the (first) Weyl algebra (see, for instance, \cite[Thm. 9.1]{Nowicki}), and so the above result characterizes their isotropy groups.
Moreover, 
if $w \in \Bbbk$,  then $\Aut_{ad_w}(A_h) = \Aut(A_h)$ since $ad_w=0$ and $\rho(w) = w$.

\medbreak 

 Recall that $\der(R)$ is an $\Aut(R)$-module via $\rho \cdot d = \rho d \rho^{-1}$, for $\rho\in \Aut(R)$ and $d\in \der(R)$. 
Since $\rho \, ad_w \, \rho^{-1} = ad_{\rho(w)}$, the set of inner derivations $\Inn(R) = \{ad_w \ | \ w \in R \}$ is an $\Aut(R)$-submodule. 

Clearly $\Aut_{d_1}(R) \cap \Aut_{d_2}(R) \subseteq \Aut_{d_1+d_2}(R)$, for any $d_1, d_2 \in \der(R)$.
The next proposition gives a condition when the equality holds.

\begin{lemma}\label{independ}
    Let $\mathcal{D}_1$ and $\mathcal{D}_2$ be two independent $\Aut(R)$-submodules of $\der(R)$.
    For any $d_1 \in \mathcal{D}_1$, $d_2 \in \mathcal{D}_2$, it holds $\Aut_{d_1+d_2}(R)=\Aut_{d_1}(R) \cap \Aut_{d_2}(R)$.
\end{lemma}

\begin{proof}
Consider $\rho \in \Aut_{d_1+d_2}(R)$. 
Then 
$$d_1+d_2 = \rho (d_1+d_2)\rho^{-1} = \rho d_1\rho^{-1} + \rho d_2 \rho^{-1}.$$

Since $\mathcal{D}_1$ and $\mathcal{D}_2$ are two independent $\Aut(R)$-submodules, it holds that $\rho \, d_1 \, \rho^{-1} = d_1$ and $\rho \, d_2 \, \rho^{-1}=d_2$.
\end{proof}

\medbreak

For the next results, we assume that $\deg(h) = N\geq 1$, $w \in A_h \backslash \Bbbk$ and recall \eqref{aut_Ah}: $\Aut(A_{h})=\{\sigma_{r(x)}\circ\tau_{a} \ | \ h(ax)=a^Nh(x)\}$, where $r(x)\in\Bbbk[x]$ and $a\in\Bbbk^{*}$, $\sigma_{r(x)}, \tau_a \in \Aut(A_{h})$ are defined as  $\sigma_{r(x)}(x)=x$, $\sigma_{r(x)}(t)=t+r(x)$, $\tau_{a}(x) = ax$ and $\tau_{a}(t) = a^{N-1}t$.

\begin{lemma} \label{lemadelta}
Consider the subset 
$$\Delta(A_h) = \{\Delta_{s(x)} \ | \ s(x) \in \Bbbk[x] \textrm{ and } \deg(s(x)) < \deg(h(x)) \}$$
of $\der(A_h)$.
Then, $\Delta$ is an $\Aut(A_h)$-submodule of $\der(A_h)$.
Moreover, $\Inn(A_h)$ and $\Delta(A_h)$ are independent $\Aut(A_h)$-submodules of $\der(A_h)$.
\end{lemma}

\begin{proof}
Let $\rho \in \Aut(A_h)$.
Then, there exist $a \in \Bbbk^*$ and $r(x) \in \Bbbk[x]$ such that $\rho=\sigma_{r(x)} \circ \tau_a$, and so $\rho^{-1} = \sigma_{-a^{N-1}r(a^{-1}x)} \circ \tau_{a^{-1}}$.
Note that, $\rho \, \Delta_{s(x)} \, \rho^{-1} (x) = 0$ and $\rho \, \Delta_{s(x)} \, \rho^{-1}(t) = a^{1-N}s(ax)$.
If one writes $s(x) = \sum_{i=0}^k b_ix^i$, where $k < N$, we get $a^{1-N}s(ax) = a^{1-N} \sum_{i=0}^k  b_i a^i x^i$, \emph{i.e.}, $\rho \Delta_{s(x)} \rho^{-1} = \Delta_{S(x)}$, where $S(x) = a^{1-N} \sum_{i=0}^k  b_i a^i x^i$.

Now we prove that $\Inn(A_h)$ and $\Delta(A_h)$ are independent submodules of $\der(A_h)$.
Indeed, let $d \in \Inn(A_h) \cap \Delta(A_h)$.
Then, there exist $w \in A_h$ and $s(x) \in \Bbbk[x]$ such that $d=ad_w$ and $d=\Delta_{s(x)}$.
Thus, 
$$0 = \Delta_{s(x)}(x) = d(x) = ad_w(x) = wx - xw,$$
that is $w \in \Bbbk[x]$.
On the other hand, 
$$s(x) = \Delta_{s(x)}(t) = d(t) = ad_w(t) = wt - tw.$$
Since $w=w(x) \in \Bbbk[x]$, it means $w(x)t = tw(x)+s(x)$ and  by \eqref{comm_t_fx_eq}, $w(x)t = w(x)t + w'(x)h(x) + s(x)$, \emph{i.e.}, $s(x) = -w'(x)h(x)$.
As $\deg(h(x))\geq 1$ and $\deg(s(x))<\deg(h(x))$, we conclude that $w \in \Bbbk$ and $s(x)=0$, and therefore $d=0$.
\end{proof}

The proof of Lemma \ref{lemadelta} also ensure that $\rho=\sigma_{r(x)}\circ\tau_a \in \Aut_{\Delta_{s(x)}}$ if and only if $s(x)=S(x)$, that is, $s(x) = a^{1-N}s(ax)$.
Then, it holds
 \begin{equation}\label{cond_s}
    \Aut_{\Delta_{s(x)}}(A_h)=\left\{\rho=\sigma_{r(x)}\circ\tau_a \in \Aut(A_h) \ | \ s(ax)=a^{N-1}s(x)\right\}.
\end{equation}

The previous lemmas lead to the following consequence.

\begin{corollary}\label{pro_grupoisotro_1}
     Let $\rho=\sigma_{r(x)}\circ\tau_a \in \Aut(A_h)$ and $D=ad_w+\Delta_{s(x)} \in \der(A_h)$.
    Then, 
\begin{equation*}
    \Aut_D(A_h) =\{\rho \in \Aut(A_h) \ | \ \rho(w)-w\in \Bbbk, \ s(ax)=a^{N-1}s(x)\}.
\end{equation*}
\end{corollary}

Therefore, the above result characterizes the isotropy groups for any derivation of $A_h$, where $h(x)$ is square free (see \cite[Thm. 10.1]{Nowicki} or Thm. \ref{nowicki_mdc1}).

\medbreak

Now, we are going to describe explicitly the group $\Aut_D(A_h)$, where $D=ad_w + \Delta_{s(x)}$.
Since $\Aut_D(A_h) \subseteq \Aut(A_h)$, we already know by Corollary \ref{carac_Aut} that if $h(x) \neq x^N$, $N \geq 1$,  then
$\Aut_D(A_h) \subseteq \Aut(A_h) = \Bbbk[x] \rtimes \tau_{\mathbb{P}}$,
where $\tau_{\mathbb{P}}$ is a finite subgroup of $\Bbbk^\ast$ (in particular, it is a finite cyclic group). 
On the other hand, Corollary \ref{carac_Aut} also ensures that if $h(x)=x^N, N \geq 1$, then $\Aut(A_h)=\Bbbk[x] \rtimes \Bbbk^\ast$.
    Thus, in this latter case, the condition $s(ax)=a^{N-1} s(x)$ in \eqref{cond_s} gives rise to two cases:
    \begin{itemize}
        \item[i)] if $s(x)= c x^{N-1}$, $c \in \Bbbk$, then $\Aut_{\Delta_{s(x)}}(A_h) = \Aut(A_h) = \Bbbk[x] \rtimes \Bbbk^\ast;$
        \item[ii)] if $s(x) \neq c x^{N-1}$, then $\Aut_{\Delta_{s(x)}}(A_h) = \Bbbk[x] \rtimes \mathbb{G}_n \subsetneq \Aut(A_h) = \Bbbk[x] \rtimes \Bbbk^\ast$, for a determined $n \geq 1$, as in Lemma \ref{lemma_raizdaunidade}.
    \end{itemize}

By the above, $\Aut_{\Delta_{s(x)}}(A_h) \subseteq \Bbbk[x] \rtimes \mathbb{G}_n$, for some $n \geq 1$, in most cases.
In fact, this does not hold only if $h(x)=x^N$ and $s(x) = cx^{N-1}$, $N \geq 1$.

\medbreak

Now, we shall see that, if $N \geq 2$, then for $D=ad_w+\Delta_{s(x)}$ with $w \in A_h\backslash \Bbbk$, it holds that $\Aut_D(A_h) \subseteq \Bbbk[x] \rtimes \mathbb{G}_n$, for some $n \geq 1$, even though $h(x) = x^N$ (and so $\tau_{\mathbb{P}} = \Bbbk^\ast$).
In particular, in this case, $\Aut_D(A_h) \subsetneq \Aut(A_h)$.
We prove this below and the case $h(x)=x$ will be considered separately later (see Proposition \ref{lema_finito_h_igual_x}).

\begin{lemma}\label{lema_finito_hmaior2}
    Let $D=ad_w+\Delta_{s(x)}$ be a derivation of $A_h$, where $\deg(h)\geq 2$ and $w \in A_h\backslash \Bbbk$.
    Then,  $\Aut_D(A_h) \subseteq \Bbbk[x] \rtimes \mathbb{G}_n$, for some $n \geq 1$.
\end{lemma}
\begin{proof}
By Corollary \ref{carac_Aut}, it is sufficient to analyze the case $h(x)=x^N$, $N\geq2$.
     Considerer $\rho = \sigma_{r(x)} \circ \tau_a \in \Aut_D(A_h)$
and write $w=\sum_{i=0}^\ell w_i(x) t^i$.
By Corollary \ref{pro_grupoisotro_1}, it holds $\rho(w)=w+c$, for some $c \in \Bbbk$.

To prove the result, we need to consider two cases.
First, if $\ell\geq 1$, then
\begin{align*}
& \rho(w) = w + c  \\
\Longleftrightarrow & \rho\left(\sum_{i=0}^\ell w_i(x) t^i \right) = \sum_{i=0}^\ell w_i(x) t^i + c \\
\Longleftrightarrow & \sum_{i=0}^\ell \rho(w_i(x)) \left(\rho(t)\right)^i = \sum_{i=0}^\ell w_i(x) t^i + c \\
\Longleftrightarrow & \sum_{i=0}^\ell w_i(ax)a^{i(N-1)}(t+r(x))^i = \sum_{i=0}^\ell w_i(x) t^i + c \\
\stackrel{Lem. \ref{lema_rhotnai}}{\Longleftrightarrow} & \left(\sum_{i=1}^\ell a^{i(N-1)} w_i(ax) (t^i+ir(x)t^{i-1}+g_i) \right)  + w_0(ax) \\
& = \sum_{i=1}^\ell w_i(x) t^i + w_0(x) + c.
\end{align*}
Thus, comparing the coefficients of the term of degree in $t$ equal to $\ell$ (the greatest one), we obtain
$a^{\ell(N-1)}w_\ell (ax)=w_{\ell}(x)$.
Since $N\geq 2$ and $\ell \geq 1$, by Lemma \ref{lemma_raizdaunidade} we get $a\in \mathbb{G}_n$, for some $n \geq 1$.

We consider now the case $\ell=0$.
Thus $w = w_0(x) \in \Bbbk[x]$.
Since $w \in A_h\backslash \Bbbk$, write $w=\sum_{i=0}^k c_ix^i$, $k \geq 1$ and $c_k\neq 0$, and then
    \begin{eqnarray*}
        \rho(w) = w + c 
        &\Longleftrightarrow & \sum_{i=0}^k c_i(a^i-1)x^i=c.
        %\\
        %&\Longleftrightarrow& c_i(a^i-1)=0, \forall i\geq 1 \ \mbox{and} \ c=0\\
        %&\Longleftrightarrow& a\in\cap_{i=1, c_i\neq 0}^{s}\mathbb{G}_{i} \ \mbox{and} \ c=0.
    \end{eqnarray*}
    
    Thus, in particular $a^k-1 =0$, that is, $a \in \mathbb{G}_k$. Therefore, $\Aut_D(A_h) \subseteq \Bbbk[x] \rtimes \mathbb{G}_n$, for some $n \geq 1$.
\end{proof}

For $h(x)=x$, we have the next result.
\begin{proposition}\label{lema_finito_h_igual_x}
    Consider $h(x)=x$ and $D=ad_w+\Delta_{s(x)} \in \der(A_h)$. 
    If $w=\sum_{i=0}^\ell w_i(x) t^i\in A_h\backslash \Bbbk$, where $\deg(w_\ell(x)) \geq 1$, then  $\Aut_D(A_h) \subseteq \Bbbk[x] \rtimes \mathbb{G}_n$, for some $n \geq 1$.
\end{proposition}
\begin{proof}
By the proof of Lemma \ref{lema_finito_hmaior2}, if $\ell = 0$, then we directly obtain $w =w(x) \in \Bbbk[x]$ and $a \in \mathbb{G}_{\deg(w(x))}$; otherwise, if $\ell \geq 1$, we have that $w_\ell$ must satisfy $w_\ell(ax) = w_\ell(x)$, and so by Lemma \ref{lemma_raizdaunidade} we conclude that $a \in \mathbb{G}_{\deg(w_\ell(x))}$.
\end{proof}

We would like to note that, by the two results above, we always have $\Aut_D (A_h) \subseteq \Bbbk[x] \rtimes \mathbb{G}_n$, except perhaps if $h(x)=x$ and $w=\sum_{i=0}^\ell w_i(x) t^i\in A_h\backslash \Bbbk$ with $w_\ell(x) \in \Bbbk$.
In fact, in this case, this can happen or not. For instance, see the next example.

\begin{example}\label{example_w}
Let $h(x)=x$, $w=\gamma t^2 + (\alpha x+\beta)t + p(x)$, with $\gamma\in \Bbbk^*$, $\alpha,\beta\in \Bbbk$ and $p(x)\in \Bbbk[x]$, $D=ad_w+\Delta_{s(x)} \in \der(A_h)$ and $\rho=\sigma_{r(x)}\circ\tau_a\in \Aut(A_h)$.
Applying Corollary \ref{pro_grupoisotro_1} and Lemma \ref{lema_rhotnai} we obtain necessary and sufficient
conditions for $\rho\in \Aut_D(A_h)$:
\begin{equation}\label{star1}
r(x)=\dfrac{\alpha(1-a)}{2\gamma} \, x
\end{equation}
and
\begin{equation}\label{star2}
\gamma(r'(x)x+r(x)^2)+(\alpha a x+\beta)r(x)+p(ax)=p(x)+c, \ \ \ \  c\in \Bbbk.
\end{equation}

Let
\[
p(x)=\frac{\alpha^2 x^2 + 2\alpha(\beta+\gamma)x + \nu}{4\gamma},\qquad \nu\in \Bbbk.
\]
Thus, substituting \eqref{star1} into \eqref{star2}, the terms in
$x^2$ and $x$ cancel identically and we obtain $p(ax)-p(x)=0$, so that
\eqref{star1} holds with $c=0$ for any $a\in \Bbbk^*$.
Hence, for this choice of $p(x)$, we have $\sigma_{r(x)}$ satisfies $\sigma_{r(x)}(w)=w$, and so $\sigma_{r(x)} \circ \tau_a \in \Aut_{D}(A_h)$, for any $a \in \Bbbk^*$.

Let $\alpha=0$ and $p(x)=x^n$ with $n\geq 1$. Then \eqref{star1} gives
$r(x)=0$ and \eqref{star2} reduces to $p(ax)-p(x)=(a^n-1)x^n = c\in k$; \emph{i.e.}, $a^n=1$.
Thus, the admissible parameters are precisely the elements of the finite cyclic group $
 \mathbb{G}_n \subsetneq \Bbbk^*$.

If $\alpha\neq0$ but $p(x)$ is constant, comparing coefficients in
\eqref{star2} forces $1-a^2=0$, \emph{i.e.}, \ $a=\pm1$. The value $a=-1$ is admissible
only if the linear component is also compatible (for instance, $\beta=-\gamma$), and when $n$ is even in the case $p(x)=x^n$, then $a=-1$
may be allowed.

In contrast to the case $\deg (h) \ge 2$, here the multiplicative component of the isotropy group depends crucially on the choice of $(\alpha,\beta,p(x))$, and it can range from the trivial subgroup $\{1\}$, to any finite cyclic group $\mathbb{G}_n$ and even $\Bbbk^*$.

\end{example}

\medbreak

Actually, by the proof of Lemma \ref{lema_finito_hmaior2}, even if $\deg(h(x))=1$, we have the following result. 

\begin{proposition} \label{propwx}
    Consider the derivation $D=ad_w+\Delta_{s(x)}$ of $A_h$, where $\deg(h)\geq 1$ and $w = \sum_{i=0}^k c_ix^i \in \Bbbk[x]$.
    Then, $\rho \in \Aut_D(A_h)$ if and only if $\rho(w)=w$.
    Moreover, if $w \in \Bbbk[x] \backslash \Bbbk$, then $\Aut_D(A_h) = \Bbbk[x] \rtimes \mathbb{G}_n$, for some $n \geq 1$.
  \end{proposition}

\begin{remark}
The above proposition has two interesting facts.
The first is that, if $w \in \Bbbk[x]$, then the automorphisms that are in $\Aut_D(A_h)$ are those whose $w$ is a fixed point.
The second is that, for $w \in \Bbbk[x] \backslash \Bbbk$, even if $h(x)=x^N$ and $s(x)=cx^{N-1}, c \in \Bbbk$, we obtain $\Aut_D(A_h) = \Bbbk[x] \rtimes \mathbb{G}_n \subsetneq \Bbbk[x] \rtimes \Bbbk^\ast = \Aut(A_h)$, where $n = gcd(i_1,i_2,..., i_j)$ being $\mathcal{S}_w = \{i_1,i_2,..., i_j\}$.
\end{remark}

\medbreak

Next, we work to get a better description for the isotropy groups given in Corolary \ref{pro_grupoisotro_1}: in fact, we shall prove that, in some cases, $\rho \in \Aut_D(A_h)$ if and only if $w$ is a fixed point for $\rho$, more than $\rho (w) - w \in \Bbbk$, as in the above proposition.
The fact that there is a fixed point plays an interesting role in the case of polynomial algebras: for a simple derivation $d \in \der(\Bbbk[x_1,\ldots,x_n])$, if $\rho \in \Aut_d(\Bbbk[x_1,\ldots,x_n])$ has a fixed point, and so an invariant maximal ideal, then $\rho = \rm{id}$ (see \cite[Prop. 7]{B2016}). 

First, we prove that such a condition is equivalent for the automorphism $\tau_a \in \Aut(A_h)$.
  \begin{proposition}\label{rho_eq_tau}
     Consider $w \in A_h$, $\deg(h)\geq 1$, and $\tau_a \in \Aut(A_h)$, for some $a \in \Bbbk^\ast$.
    Then $\tau_a(w)-w \in \Bbbk \Longleftrightarrow \tau_a(w)=w$.
    \end{proposition}
    \begin{proof}
Let $c \in \Bbbk$ and write $w=\sum_{i=0}^\ell w_i(x) t^i$.
Then 
\begin{align*}
& \tau_a(w) = w + c  \\
\Longleftrightarrow & \tau_a\left(\sum_{i=0}^\ell w_i(x) t^i \right) = \sum_{i=0}^\ell w_i(x) t^i + c \\
\Longleftrightarrow & \sum_{i=0}^\ell \tau_a(w_i(x)) \left(\tau_a(t)\right)^i = \sum_{i=0}^\ell w_i(x) t^i + c \\
\Longleftrightarrow & \sum_{i=0}^\ell a^{i(N-1)} w_i(ax) t^i = \sum_{i=0}^\ell w_i(x) t^i + c.
\end{align*}
Thus, comparing the coefficients of the term of degree in $t$ equal to zero, we obtain
$w_0 (ax)=w_0(x)+c$.
%If $w_0(x)=0$, then $c=0$. Otherwise, 
Write $w_0(x)=\sum_{i=0}^{m}c_ix^i$.
Thus $w_0 (ax)=w_0(x)+c $ if and only if $\sum_{i=0}^{m}c_i a^i x^i = \sum_{i=0}^{m}c_ix^i + c$, and again, analyzing the coefficient of degree zero, we obtain $c_0 = c_0+c$, and so $c=0$.
\end{proof}

Now, we investigate the situation for a general $\rho=\sigma_{r(x)} \circ \tau_a \in \Aut_D(A_h)$.

\begin{proposition}\label{prop_rhowigualaw}
Let $D=ad_w+\Delta_{s(x)}$ be a derivation of $A_h$, $\deg(h)\geq 1$, where $\deg_{t}(w)\geq 2$, and $\rho \in \Aut_D(A_h)$.
If $\Aut_D(A_h) \subseteq \Bbbk[x] \rtimes \mathbb{G}_n$, $n \geq 1$, then $\rho(w)-w\in \Bbbk$ if and only if $\rho(w) = w$.
\end{proposition}

\begin{proof}
    Consider the exact sequence of groups
    $$\Bbbk[x] \stackrel{\iota}{\longrightarrow} \Bbbk[x] \rtimes \mathbb{G}_{n} \stackrel{\pi}{\longrightarrow} \mathbb{G}_{n}$$
where $\iota(r(x)) = \sigma_{r(x)}$ and $\pi(\sigma_{r(x)} \circ \tau_a) = \tau_a$.
In particular, for $\rho = \sigma_{r(x)} \circ \tau_a \in \Aut_D(A_h) \subseteq \Bbbk[x] \rtimes \mathbb{G}_{n}$, it holds $\rho^n = \sigma_{R_n(x)} \circ \tau_{a^n} = \sigma_{R_n(x)}$ by Lemma \ref{rho_n}, where $R_n(x) = \sum_{i=0}^{n-1} a^{i(1-N)} r(a^i x)$.
If $\rho(w)=w+c$, for some $c \in \Bbbk$, then $\rho^i(w) = w + i c$, and so $ \rho^n(w) =  w + n c$  if and only if $\sigma_{R_n(x)}(w) = w + nc.$

Write $w=\sum_{i=0}^\ell w_i(x) t^i$, $\ell \geq 2$.
Then 
\begin{align*}
& \sigma_{R_n(x)}(w) = w + nc  \\
    \Longleftrightarrow & \sigma_{R_n(x)}\left(\sum_{i=0}^\ell w_i(x) t^i \right) = \sum_{i=0}^\ell w_i(x) t^i + nc \\
\Longleftrightarrow & \sum_{i=0}^\ell w_i(x) \left(\sigma_{R_n(x)}(t)\right)^i = \sum_{i=0}^\ell w_i(x) t^i + nc \\
\Longleftrightarrow & \sum_{i=0}^\ell w_i(x) (t+R_n(x))^i = \sum_{i=0}^\ell w_i(x) t^i + nc \\
\stackrel{Lem. \ref{lema_rhotnai}}{\Longleftrightarrow} & \left(\sum_{i=1}^\ell w_i(x) \left(t^i+iR_n(x)t^{i-1}+g_i \right) \right)  + w_0(x) \\
& = \sum_{i=1}^\ell w_i(x) t^i + w_0(x) + nc.
\end{align*}
The coefficient of the term of degree in $t$ equal to $\ell-1$ gives us the following relation:
$w_{\ell-1}(x) + \ell R_n(x) = w_{\ell -1} (x)$.
Thus, $R_n(x)=0$ and so $\rho^n=id$.
Consequently $c=0$, and so $\rho(w)=w$.
\end{proof}

In light of Lemma \ref{lema_finito_hmaior2} and Proposition \ref{prop_rhowigualaw}, we obtain the following results. 

\begin{corollary}
Let $\deg(h)\geq 2$ and $\rho \in \Aut_D(A_h)$, where $D=ad_w+\Delta_{s(x)}$.
Then $\rho(w)-w\in \Bbbk$ if and only if $\rho(w) = w$.
\end{corollary}

\begin{corollary}
Let $h(x)=x$ and $D=ad_w+\Delta_{s(x)} \in \der(A_h)$. 
If $\rho \in \Aut_D(A_h) \subseteq \Bbbk[x] \rtimes \mathbb{G}_n$, for some $n \geq 1$, then  $\rho(w)-w\in \Bbbk$ if and only if $\rho(w) = w$.
\end{corollary}

 Note that, in the results above, we always have that $\rho(w)-w\in\Bbbk$ is equivalent to condition $\rho(w)=w$, except perhaps if $h(x)=x$ and $\tau_a \in \Aut_D(A_h)$, for all $a \in \Bbbk^\ast$.
In this case, this can happen or not, as we see in the next example.

\begin{example}
   Consider $h(x)=x$, $w=t+x$, $D=ad_w+\Delta_{s(x)}$, where $s(x)=
   \alpha\in\Bbbk$, and
$\rho=\sigma_{r(x)}\circ\tau_a\in \Aut(A_h)$.
Then,
$$\rho(w)-w=t+r(x)+ax-t-x=r(x)+(a-1)x.$$
Thus, if $r(x)=(1-a)x+\beta$, $\beta\in\Bbbk$, then $\rho(w)-w = \beta \in\Bbbk$.
Therefore, $\rho(w) = w$ if and only if $\beta =0$. 
\end{example}

The next result presents a situation where $\Aut_D(A_h) \subseteq \tau_{\mathbb{P}} \subseteq \Bbbk^\ast$ (as in \eqref{tau_P}).
Afterward, as a consequence, in this case, the condition $\tau_a(w)-w \in \Bbbk$ is equivalent to $w$ being a fixed point of $\tau_a$ (see Proposition \ref{rho_eq_tau}), and it characterizes $\Aut_D(A_h)$.

\begin{proposition}\label{prop_rx_zero}
  Let $\deg(h)\geq 1$, $w=\sum_{i=0}^\ell w_i(x)t^i\in A_h$, where $\ell \geq 2$, $w_\ell (x) \neq 0$ and $w_{\ell-1} (x) = 0$,  and $D=ad_w + \Delta_{s(x)}$.
    If $\rho = \sigma_{r(x)} \circ \tau_a \in \Aut_D(A_h)$, then $r(x)=0$.
\end{proposition}
\begin{proof}
    By Proposition \ref{pro_grupoisotro_1}, we have $\rho(w)-w=c\in \Bbbk$. Then, analogously to the proof in Lemma \ref{lema_finito_hmaior2}, we have 
    $\rho(w)=w+c$ if and only if
    \begin{equation*}
 \sum_{i=1}^\ell a^{i(N-1)}w_i(ax)(t^i+ir(x)t^{i-1}+g_i)+w_0(ax) =\sum_{i=0}^\ell w_i(x)t^i+c.
    \end{equation*}

Thus, comparing the coefficients of the terms of degree in $t$ equal to $\ell-1$ and $\ell$, we obtain
$\ell a^{\ell(N-1)}w_\ell (ax)r(x)+a^{(\ell-1)(N-1)}w_{\ell-1}(ax)=w_{\ell-1}(x)$
and 
$a^{\ell(N-1)}w_\ell(ax)=w_\ell(x),$ respectively. 
Since $w_{\ell-1}(x)=0$, the first equality means $a^{\ell(N-1)}w_\ell(ax)r(x)=0$, and
 as $w_\ell(x) \neq 0$ we conclude that $r(x)=0$.
\end{proof}

\medbreak

 In the conditions of the above proposition, if $\rho \in \Aut_D(A_h)$ then $\rho = \tau_a$, for some $a \in \Bbbk^\ast$.
 Hence, $\Aut_{D}(A_h) = \{ \tau_a \in \tau_{\mathbb{P}} \ | \ \tau_a(w)=w\}.$
 In particular, $\Aut_{D}(A_h) \subseteq \mathbb{G}_n$, for some $n \geq 1$, as long as $\deg(h) \geq 2$ (see Lemma \ref{lema_finito_hmaior2}), or $\deg(h)=1$ and $w=\sum_{i=0}^\ell w_i(x) t^i$, $\ell\geq 2$, $\deg(w_\ell(x))\geq 1$ (see Proposition \ref{lema_finito_h_igual_x}).
  In particular, let $h(x)=x$, $w = xt^2$ and $\rho \in \Aut_D(A_h)$; by the above proposition $\rho = \tau_a$, for some $a \in \Bbbk^\ast$, and by Proposition \ref{rho_eq_tau} it holds $\tau_a(w)=w$ and so $a=1$, \emph{i.e.}, in this particular case, $\Aut_D(A_h)=\{Id\}$.
 
However, if
 $\deg(h)=1$ and $w=\sum_{i=0}^\ell w_i(x) t^i$, with $w_\ell(x)\in\Bbbk$ and $w_{\ell-1}=0$, it can happen that $\Aut_D(A_h) = \Bbbk^\ast$.
 For example, if $w=\sum_{i=0}^\ell c_i t^i$ with $c_i\in \Bbbk$ for all $i$, and $c_{\ell-1}=0$, then $\tau_a(w) = w$, for all $a\in\Bbbk^*$.

  \medbreak

As seen in Proposition \ref{prop_rx_zero}, if $w=\sum_{i=0}^\ell w_i(x)t^i\in A_h$ with $\ell \geq 2$ and $w_{\ell-1} (x) = 0$, then $\Aut_D(A_h) \subseteq \Bbbk^\ast$.
  However, it is not true when $w_{\ell-1} (x) \neq 0$ (see Example \ref{example_w}).
On the other hand, by Proposition \ref{propwx}, the case when $w \in \Bbbk[x] \backslash \Bbbk$, we already know $\Aut_D(A_h)=\Bbbk[x] \rtimes \mathbb{G}_n$, for some $n\geq 1$.
Below, we will analyze the case in which $w \in A_h\backslash \Bbbk$ with $\deg_{t}(w)=1$.

\begin{lemma} \label{lema_wlinear} Let $w= f(x)t+g(x), \ f(x), g(x)\in \Bbbk[x]$ and $D=ad_w + \Delta_{s(x)}$. 
    If $\rho = \sigma_{r(x)} \circ \tau_a \in \Aut_D(A_h)$, then $f(x)=a^{N-1}f(ax)$ and $f(x)r(x)=g(x)-g(ax)+c$, where $c = \rho(w) - w \in \Bbbk$.
\end{lemma}
\begin{proof}
    By Proposition \ref{pro_grupoisotro_1}, $\rho(w)-w=c\in \Bbbk$. Then,
\begin{eqnarray*}
    \rho(w)-w=c&\Longleftrightarrow&a^{N-1}f(ax)(t+r(x))-f(x)t-g(t)=c\\
    &\Longleftrightarrow& a^{N-1}f(ax)=f(x) \  \mbox{and} \ f(x)r(x)=g(x)-g(ax)+c.
    \end{eqnarray*}
\end{proof}

\begin{remark}
    Let $D=ad_w + \Delta_{s(x)} \in \der(A_h)$, where $w = f(x)t +g(x)$, and $\rho = \sigma_{r(x)} \circ \tau_a \in \Aut_D(A_h)$.
    If $f(x)=\sum_{i=0}^k c_ix^i\in \Bbbk[x]$, by Lemmas \ref{lemma_raizdaunidade} and \ref{lema_wlinear}, we get that $a\in \cap_{i \in \mathcal{S}_f}\mathbb{G}_{N-1-i}$, where $\mathcal{S}_f=\{i \ | \ c_i\neq 0\}$.
    Then:
    \begin{itemize}
  \item[(i)] If $g(x)=0$ and $f(x) = f \in \Bbbk^\ast$, then $a^{N-1}=1$ and $r(x) = c/f\in \Bbbk$.
  Hence, $\Aut_D(A_h) \subseteq  \langle \sigma_{c/f} \rangle \rtimes \mathbb{G}_{N-1-i}$.
  
     \item[(ii)] If $\deg(f(x))\leq \deg(g(x))$, then $r(x)=\dfrac{g(x)-g(ax)+c }{f(x)}$ and 
    
    $\Aut_D(A_h)\subseteq \langle \sigma_{r(x)} \rangle \rtimes \left(\cap_{i \in \mathcal{S}_h}\mathbb{G}_{N-1-i}\right)$.

        \item[(iii)] If $\deg(f(x))>\deg(g(x))$, then $r(x)=0$ and
        $\Aut_D(A_h)\subseteq \cap_{i \in \mathcal{S}_h}\mathbb{G}_{N-1-i}$.
    \end{itemize}
\end{remark}

Below we provide some examples that better illustrate these cases above.

\begin{example}
    \begin{enumerate}
         \item If $w=t$, then $a^{N-1}=1$ and $r(x)\in \Bbbk$.
    \item If $w=t+x$, then $a^{N-1}=1$ and $r(x)=x(1-a)+c, \ c\in \Bbbk$.
    \item If $w=xt+x$, then $a^{N}=1$, $c=0$ and $r(x)=1-a$.
      \item If $w=xt$, then $a^N=1$ and $r(x)=0$.
      \item If $w=x^2t+x$, then $a=1$ and $r(x)=0$.
        \end{enumerate}
\end{example}

\medbreak

 In the above examples, it always happen that if $\rho \in \Aut_D(A_h)$, then $\rho = \sigma_{r(x)} \circ \tau_a$, where $a$ is a root of unity and the polynomial $r(x)$ is determined, except maybe for translation by constants.
In the next example, the situation is totally different.
\begin{example}
    Let $h(x)=x^3 + \alpha \, x + \beta$, where $0 \neq \alpha \,\beta \in \Bbbk$, $D=ad_w + \Delta_{s(x)}$ with $w, s(x) \in \Bbbk[x]$, $\deg(s(x))<3$.
    Then, $\Aut_D(A_h) = \Bbbk[x]$.
    Indeed, by Corollary \ref{carac_Aut} ii) it follows that $\Aut(A_h) = \Bbbk[x]$, and since $D(t) \in \Bbbk[x]$ and $D(x)=0$, it follows that $\sigma_{r(x)} \circ D = D \circ \sigma_{r(x)}$, that is, $\Aut_D(A_h) = \Aut(A_h) = \Bbbk[x]$.
\end{example}

It is known that any locally nilpotent derivation of $A_h$ is, up to automorphism, of the form: $D_{p(x)}$, where $D_{p(x)}(x)=0$ and $D_{p(x)}(t)=p(x)$ (see \cite[Prop. 2]{ISF2021}). For this family of derivations we have the following result.

\begin{remark} Let $D_{p(x)}\in LND(A_h)$, $p(x)\in \Bbbk[x]$ and $\sigma_{r(x)} \circ \tau_a \in \Aut(A_h)$.
Then
    $$\Aut_{D_{p(x)}}(A_h) =\{\sigma_{r(x)}\circ\tau_{a} | \ p(ax)=a^{N-1}p(x)\}.$$

   In this case, we can uniquely decompose the derivation $D_{p(x)}$ as $ad_w+\Delta_{s(x)}$ which will depend on the degree of the polynomial $p(x)$ in relation to the degree of $h(x)$. 

If $\deg(p(x))<\deg(h(x))$, then $D_{p(x)}=\Delta_{p(x)}$. Otherwise, by the division algorithm,  there exists $p_1(x), p_2(x)\in \Bbbk[x]$, such that $p(x)=p_1(x)h(x)+p_2(x)$, with $\deg(p_2(x))<\deg(h(x))$.
Thus, $D_{p(x)}=ad_w+\Delta_{p_2(x)}$, where $ w'=-p_1(x)$.
Indeed, as $ad_w(x)=0$, then $w\in \Bbbk[x]$, and so $ad_w(t)=p_1(x)h(x)$ if and only if $p_1(x)=-w'$.
\end{remark}

\section{The singular case $\gcd(h,h') \neq1$}\label{section5}

In the previous section, we dealt with a derivation $D=ad_w + \Delta_{s(x)}$ of $A_h$.
If $\gcd(h,h')=1$, any $D \in \der(A_h)$ has such decomposition \cite[Thm. 10.1]{Nowicki} (see Thm. \ref{nowicki_mdc1}).
When $h$ has multiple roots, new derivations appear, encoded by \emph{special polynomials}, forcing a refinement of the isotropy analysis.

\medbreak

Let $\deg(h)=N \geq 2$ and let $r$ be the number of all pairwise different roots of the polynomial $h(x)$ (recall that we assume $\Bbbk$ an algebraically closed field of characteristic zero). Denote by $\psi$ the greatest common divisor of $h(x)$ and $h'(x)$, and in this case $\deg (\psi) = N-r > 0$. 

We say that a polynomial $H \in A_h$ is \emph{special} if it is of the form 
$$ H=h_nt^n+\ldots+h_1t^1,$$
where $n \ge 1$, $h_i \in \Bbbk[x]$ and $\deg (h_i) < N-r$, for all $i= \I_{1,n}$.
By convention, the zero polynomial is also special.
Given a special polynomial $H$, the linear map $E_H: A_h \to A_h$ defined as
$$ E_H(f)=\left[ \tfrac{1}{\psi}H,f \right], \quad f\in A_h, $$
is a well-defined derivation of $A_h$ (see \cite[Prop. 2.1 and 8.2]{Nowicki}). 
Moreover, $E_H$ is an inner derivation of $A_h$ if and only if $H=0$.

\begin{theorem}[{\cite[Theorem 11.1]{Nowicki}}]
Let $h\in \Bbbk[x]$, $\deg(h) \ge 2$ and $\psi=\gcd(h,h')$. 
Then every derivation $D\in\der(A_h)$ can be uniquely expressed as
\[
D=ad_w+E_H+\Delta_{s(x)},
\]
where $w\in A_h$, $s(x)\in \Bbbk[x]$ with $\deg (s) <\deg (h)$, 
and $H$ is a special polynomial as defined above. Moreover, $D$ is inner if and only if $H=0$ and $s(x)=0$.
\end{theorem}

\medbreak 

Recall that $A_h = \Bbbk [x] [t ; d]$ where $d = h(x) \, \partial_x$.
Let $S=\{\psi^n \mid n\ge 0\}\subseteq \Bbbk[x]$. Since $d(\psi)=\psi' \, h(x) \in \psi \, \Bbbk[x]$, the set $S$ is an Ore set in $A_h$ (see \cite[Lemma 4.2]{BLO2015}).
We denote by
$$R_S:=\Bbbk[x][S^{-1}]
\quad\text{and}\quad
\mathcal B:=A_h[S^{-1}]\cong R_S[t;d_S],$$
where $d_S:R_S\to R_S$ is the unique extension of $d$.
Moreover, for any $\rho \in \Aut(A_h)$, we can also consider its extension to $\mathcal{B}$.
Indeed, if $\rho = \sigma_{r(x)} \circ \tau_a \in \Aut(A_h)$, then $h(x)$ satisfies $h(ax)=a^N h(x)$, and consequently $h^\prime(ax)=a^{N-1} h^\prime(x)$.
As $\psi(x) = \gcd(h,h^\prime)$, then $\psi(ax)$ divides both $h(x)$ and $h^\prime(x)$, and so $\psi(ax) \, | \, \psi(x)$, \emph{i.e.}, $\psi(ax) = a_\psi \, \psi(x)$, for some $a_\psi \in \Bbbk$.
Thus, $\rho(\psi)=\psi(ax)=a_\psi \, \psi(x)$ and, therefore, there exists a unique extension of $\rho$ to an automorphism of $\mathcal{B}$ (which we will also denote by $\rho$).
Furthermore, for an element $u \in A_h$, by $\frac{u}{\psi} \in \mathcal{B}$ we will mean $\frac{1}{\psi}u \in \mathcal{B}$.

\medbreak

Consider the $\Bbbk$-subspaces $\Inn(A_h)$ and $\Delta(A_h)$ as in the previous section.
If $H_1$ and $H_2$ are two special polynomials, and $c \in \Bbbk$, then $H_1 + c H_2$ is a special polynomial and $E_{H_1}+ c E_{H_2} = E_{H_1 + c H_2}$.
Thus, one can consider the $\Bbbk$-subspace $\mathcal{E}(A_h)= \{ E_H \ | \ \textrm{$H$ is a special polynomial of $A_h$} \}$, and the above theorem ensure that $\der(A_h) = \Inn(A_h) \oplus \Delta(A_h) \oplus \mathcal{E}(A_h)$ as $\Bbbk$-spaces.

We would like to conclude that $\Inn(A_h), \Delta(A_h)$ and $\mathcal{E}(A_h)$ are independent $\Aut(A_h)$-submodules of $\der(A_h)$, and, if that were the case, then $\Aut_D(A_h),$ $D=ad_w+\Delta_{s(x)}+E_H$, could be calculated using Lemma \ref{independ} by computing each of the groups $\Aut_{ad_w}(A_h),$ $\Aut_{\Delta_{s(x)}}(A_h)$ and $\Aut_{E_H}(A_h)$ separately.
But that is not the case. 
In general, the $\Bbbk$-space $\mathcal{E}(A_h)$ is not even an $\Aut(A_h)$-submodule of $\der(A_h)$.
Indeed, let $\rho \in \Aut(A_h)$.
Since $\rho \circ E_H \circ \rho^{-1} (f) = \left[ \tfrac{\rho(H)}{a_\psi \psi}, f \right]$ for all $f \in A_h$, if $\rho(H)$ is a special polynomial, then $\rho \circ E_H \circ \rho^{-1} = E_{(a_\psi)^{-1}\rho(H)}$;
but it is not always the case that $\rho(H)$ is a special polynomial.
More precisely, conjugating a derivation of type $E_H$ may not produce a derivation $E_J \in \mathcal{E}(A_h)$. 
Consider the following example.

\begin{example}
Let $h(x) = x^2$ and $\rho=\sigma_{x \, p(x) + c } \circ \tau_a \in \Aut(A_h)$, $p(x)\in\Bbbk[x]$ and $a, c\in\Bbbk$, and $H=t \in A_h$.
Note that $\psi=x$ and $E_t=[x^{-1}t,-]$.
Since $\rho(t)=a(t+x p(x)+c)$ and $\rho(x)=ax$, we obtain
$$ \rho(x^{-1}t)=\rho(x)^{-1}\rho(t)
=x^{-1}t+p(x)+c\,x^{-1}.$$

Since $[cx^{-1},x]=0$ and $[cx^{-1},t]=c$, it means that $[cx^{-1},-]=\Delta_c$.
Then
$$
\rho\circ E_t\circ\rho^{-1}
=
[x^{-1}t+p(x)+c\,x^{-1},-]
=
E_t+ad_{p(x)}+\Delta_c.$$
\end{example}

\medbreak

The above discussion ensures that for a derivation $D=ad_w+\Delta_{s(x)}+E_H$, one cannot consider the derivation $E_H$ separately.
In this context, we denote $w^\ast := w + \tfrac{H}{\psi} \in \mathcal{B}$,
so that we can write $ad_w + E_H = [w^\ast, - ]$.
\begin{comment}
    Denote by $ W = Ad \oplus E$ as $\Bbbk$-spaces.
The following proposition shows that this component is independent from the $\Delta$, as $\Aut(A_h)$-submodules of $\der(A_h)$.
%differential part $\Delta_{s(x)}$.
\end{comment}

\begin{lemma}\label{R_S} If $u \in \mathcal B$ satisfies $[u,x]=0$, then $u \in R_S$.
\end{lemma}

\begin{proof}
Let $u=\sum_{i=0}^m f_i t^i$ where $f_i\in R_S$ and $f_m \neq 0$.
If $m\ge1$, then
$$[u,x]=\sum_{i=1}^m f_i[t^i,x].$$

By Lemma \ref{lematix}, the highest power of $t$ of $[t^i,x]$ is $ih(x)t^{i-1}$, hence the highest power of $t$ of $[u,x]$ is $mf_m h(x)t^{m-1}$. Since $R_S$ is a domain and $h(x)\neq0$, this term is nonzero, contradicting $[u,x]=0$. Therefore, $m=0$, \emph{i.e.}, $u\in R_S$.
\end{proof}

\begin{lemma} \label{Kerds}Let $d_S:R_S\to R_S$ be the unique extension of the derivation $d=h(x)\partial_x$, where $h(x)\in\Bbbk[x]\setminus\{0\}$. Then, $\ker(d_S)=\Bbbk$.    
\end{lemma}

\begin{proof}
 Let $u\in R_S$ and assume that $d_S(u)=0$.
Since $d_S(u)=h(x)u'$, we obtain $u'=0$.
Write $u=\frac{f}{g} \in R_S$, that is, $f\in\Bbbk[x]$ and $g = \psi^n \in S \subseteq\Bbbk[x]$.
If $f=0$, then $u=0$, \emph{i.e.,} $u \in \Bbbk$.
Assume $f \neq 0$.
Note that $u=\frac{f}{g} = \frac{F \Psi}{G \Psi} = \frac{F}{G} \in \Bbbk(x)$, where $F, G \in \Bbbk[x]$, $\gcd(f,g) = \Psi$, $\gcd(F,G)=1$, and $\Bbbk(x)$ denotes the field of fractions of $\Bbbk[x]$, which coincides with the field of fractions of $R_S$.
In particular, for $F \neq 0$, $\deg(F^\prime) < \deg(F)$.
Since $0 = u ^\prime = \frac{G F^\prime - F G^\prime}{G^2}$, it follows that $G F^\prime = F G^\prime$.
Hence, $F \mid G F^\prime$, and so $F \mid F^\prime$ since $\gcd(F,G)=1$.
If $F^\prime \neq 0$, then $\deg(F) \leq \deg(F^\prime)$, but it is a contradiction.
Therefore, $F^\prime = 0$, \emph{i.e.}. $F \in \Bbbk$.
Then, $F G^\prime =G F^\prime = 0$, that is $G^\prime =0$, and so $G \in \Bbbk$, and consequently $u = \frac{F}{G} \in \Bbbk$.
Thus, $\ker(d_S) \subseteq \Bbbk$, and since the reverse inclusion is clear, it follows that $\ker(d_S)=\Bbbk$.
\end{proof}

\begin{proposition} \label{classificacao}
Let $D=ad_w+E_H+\Delta_{s(x)}=[w^\ast,-]+\Delta_{s(x)} \in \der(A_h)$, where $w^\ast:=w+\psi^{-1}H \in \mathcal B$, $H$ is a special polynomial, $w\in A_h$, $\psi=\gcd(h,h')$ and $\deg (s(x)) < \deg (h(x))$.
Then, $\rho \in \Aut_D(A_h)$ if and only if $d_S(\rho(w^\ast)-w^\ast)=a^{1-N}s(ax)-s(x)$ and $\rho(w^\ast)-w^\ast \in R_S$.
\end{proposition}

\begin{proof}
Consider the extension of $\rho$ to an automorphism of $\mathcal B$.
Then, it holds $\rho \circ[w^\ast,-] \circ \rho^{-1}=[\rho(w^\ast),-]$. 
Moreover, $\rho\circ\Delta_{s(x)}\circ\rho^{-1}
=\Delta_{a^{1-N}s(ax)}$, since $(\rho\circ\Delta_{s(x)}\circ\rho^{-1})(x)=0$ and $(\rho\circ\Delta_{s(x)}\circ\rho^{-1})(t)=a^{1-N}s(ax)$. Therefore,
\[
\rho\circ D\circ\rho^{-1}
=
[\rho(w^\ast),-]+\Delta_{a^{1-N}s(ax)}.
\]

Thus, $\rho\in\Aut_D(A_h)$ if and only if
\begin{equation}\label{prop_42}
[\rho(w^\ast)-w^\ast,-]+\Delta_{a^{1-N}s(ax)-s(x)}=0.
\end{equation}
Now, we prove that \eqref{prop_42} holds if and only if $d_S(\rho(w^\ast)-w^\ast)=a^{1-N}s(ax)-s(x)$ and $\rho(w^\ast)-w^\ast \in R_S$.
First, assume that \eqref{prop_42} holds.
Then, evaluating it at $x$ and $t$, we obtain
\[
[\rho(w^\ast)-w^\ast,x]=0 \quad \text{and} \quad [\rho(w^\ast)-w^\ast,t]=s(x)-a^{1-N}s(ax).
\]

By Lemma \ref{R_S}, it follows that $\rho(w^\ast)-w^\ast\in R_S$, and so the Ore relation in $\mathcal B$ gives
\[
[\rho(w^\ast)-w^\ast,t]
=
-d_S(\rho(w^\ast)-w^\ast).
\]

Therefore, $d_S(\rho(w^\ast)-w^\ast)=a^{1-N}s(ax)-s(x)$.

 For the converse, assume that $\rho(w^\ast)-w^\ast \in R_S$ and $d_S(\rho(w^\ast)-w^\ast)=a^{1-N}s(ax)-s(x)$.
Since $\rho(w^\ast)-w^\ast \in R_S$, then 
$$[\rho(w^\ast)-w^\ast,t]
=
-d_S(\rho(w^\ast)-w^\ast) = s(x) - a^{1-N}s(ax),$$
and so 
$$ [\rho(w^\ast)-w^\ast,t]+\Delta_{a^{1-N}s(ax)-s(x)}(t)=0,$$
and clearly 
$$ [\rho(w^\ast)-w^\ast,x]+\Delta_{a^{1-N}s(ax)-s(x)}(x)=0,$$
that is, \eqref{prop_42} holds.
\end{proof}

\begin{remark}
If $\gcd(h,h')=1$, \emph{i.e.} $\psi=1$, then $R_S=\Bbbk[x]$, $\mathcal B=A_h$, $w^\ast=w$ and $H=0$.
In this case, Proposition~\ref{classificacao} coincides with the square-free criterion obtained in the previous section (see Corollary \ref{pro_grupoisotro_1}).
\end{remark}

\begin{corollary}\label{adwEH} Let  $H$ be a special polynomial and $w\in A_h$. If $D=ad_w+E_H=[w^\ast,-]$, then
\[
\Aut_D(A_h)
=
\big\{\rho\in\Aut(A_h)\ \big|\ 
\rho(w^\ast)-w^\ast\in\Bbbk \big\}.
\]
\end{corollary}

\begin{proof}
Proposition \ref{classificacao}, considering $s(x)=0$, states that $\rho\in\Aut_D(A_h)$ if and only if $\rho(w^\ast)-w^\ast \in R_S$ and $d_S(\rho(w^\ast)-w^\ast)=0$;
by Lemma \ref{Kerds}, the last condition is equivalent to $\rho(w^\ast)-w^\ast \in \Bbbk$.
\end{proof}

Consider $D=E_H$ in the above result, \emph{i.e.,} $w=0$.
Then, one immediately gets the following:
\begin{corollary}
Let $H$ be a special polynomial. Then
\[
\Aut_{E_H}(A_h)
=
\bigl\{\rho\in\Aut(A_h)\ \big|\ 
\rho(\psi^{-1}H)-\psi^{-1}H\in\Bbbk
\bigr\}.
\]
\end{corollary}

\medbreak

Consider $D=ad_w+E_H$, for some $w, H \in A_h$, where $H$ is a special polynomial, and $\rho \in \Aut(A_h)$.
It is clear that if $\rho(w)-w, \rho\left( \frac{H}{\psi}\right) - \frac{H}{\psi}\in \Bbbk$, then $\rho(w^*)-w^* \in \Bbbk$, where $w^*=w + \frac{H}{\psi}$.
Therefore, $\rho \in \Aut_D(A_h)$.
The next example illustrates that $\rho(w^*)-w^* \in \Bbbk$ can happen even if $\rho(w)-w \notin \Bbbk$ and $\rho(\frac{H}{\psi})-\frac{H}{\psi} \notin \Bbbk$.
That is, to $\rho \in \Aut_D(A_h)$ it is necessary to verify $\rho\left(w + \frac{H}{\psi}\right) - \left( w + \frac{H}{\psi}\right) \in \Bbbk$, and not just $\rho(w)-w, \rho\left( \frac{H}{\psi}\right) - \frac{H}{\psi}\in \Bbbk$ separately.

\begin{example}
Consider $A_{x^2}$, $w=\frac{1}{1-a} x + c$, where $1 \neq a \in \Bbbk$, and the special polynomial $H=t$.
Thus, $\psi = x$ and $w^* = w + \frac{t}{x}$.
For $\rho = \sigma_{x^2} \circ \tau_a \in \Aut(A_h)$, we get $\rho(w)-w = -x$ and $\rho\left(\frac{t}{x}\right) - \frac{t}{x} = x$, and so $\rho(w^*)-w^* = 0 \in \Bbbk$.
Therefore, $\rho \in \Aut_{ad_w + E_H}(A_h)$, but $\rho \notin \Aut_{ad_w}(A_h)$ and $\rho \notin \Aut_{ E_H}(A_h)$.
\end{example}

\begin{remark}
Let $u\in R_S$ and assume $d_S(u)\in \Bbbk[x]$.
Then the commutator $[u,-]$ is a well-defined derivation of $A_h$ and $[u,-]=D_{-d_S(u)}$, where $D_{-d_S(u)}$ is the derivation of $A_h$ determined by $D_{-d_S(u)}(x)=0$ and $D_{-d_S(u)}(t)=-d_S(u)$.
In particular, if $\deg (d_S(u))< \deg (h(x))$, then $
[u,-]=\Delta_{-d_S(u)}$.
In general, we write $d_S(u)=q(x)h(x)+r(x)$, with $q(x),r(x)\in\Bbbk[x]$ and  $\deg(r(x))< \deg(h(x))$.
Then
\[
[u,-]=ad_f+\Delta_{-r(x)},
\]
where $f\in\Bbbk[x]$ is any polynomial that satisfies $f'(x)=q(x)$. 
\end{remark}

\section*{Acknowledgments}

The authors are grateful to Prof. Ivan Pan (CMAT-UdelaR) for valuable suggestions and comments.
The authors were partially supported by the Rio Grande do Sul Research Foundation (FAPERGS) projects n. 24/2551-0001560-4 (R.B.), 23/2551-0000913-7 (L.D.S.) and 23/2551-0000803-3 (G.M.).

\bibliographystyle{abbrv}

\bibliography{refs}

\bigskip

\noindent
\textsc{Rene Baltazar}\\
Universidade Federal do Rio Grande - FURG\\
Santo Antônio da Patrulha/RS, Brasil\\
\texttt{renebaltazar.furg@gmail.com}

\bigskip

\noindent
\textsc{Leonardo Duarte Silva}\\
Universidade Federal do Rio Grande do Sul - UFRGS\\
Porto Alegre/RS, Brasil\\
\texttt{dsleonardo@ufrgs.br}

\bigskip

\noindent
\textsc{Grasiela Martini}\\
Universidade Federal do Rio Grande do Sul - UFRGS\\
Porto Alegre/RS, Brasil\\
\texttt{grasiela.martini@ufrgs.br}

\end{document}